\documentclass[12pt, reqno]{amsart}

\usepackage{amsmath,amsfonts,amssymb,amsthm,shuffle,latexsym,mathtools,mathrsfs}
\usepackage{bm}
\usepackage{enumitem}
\usepackage[usenames, dvipsnames, svgnames]{xcolor}
\usepackage{ytableau}
\usepackage{tikz-cd}
\usepackage{pbox}
\usepackage{upgreek, dsfont, float}

\makeatletter
\@namedef{subjclassname@2020}{\textup{2020} Mathematics Subject Classification}
\makeatother

\newtheorem{theorem}{Theorem}[section]
\newtheorem{lemma}[theorem]{Lemma}
\newtheorem{proposition}[theorem]{Proposition}
\newtheorem{corollary}[theorem]{Corollary}

\theoremstyle{definition}
\newtheorem{definition}[theorem]{Definition}
\newtheorem{example}[theorem]{Example}

\theoremstyle{remark}
\newtheorem{remark}[theorem]{Remark}

\numberwithin{equation}{section}

\newcommand{\K}{\mathbb{K}}
\newcommand{\sym}{\mathfrak{S}}

\newcommand{\rw}{{\rm rw}}
\newcommand{\rwr}{{{\rm rw}_\mathcal{R}}}

\newcommand{\set}{{\rm set}}

\newcommand{\iT}{\mathscr{T}}

\newcommand{\PP}{{\rm P}}

\newcommand{\bb}{\boldsymbol}
\newcommand{\ul}{\underline}
\newcommand{\B}{{\rm B}}
\newcommand{\D}{\text{D}}

\newcommand{\SIT}{\mathrm{SIT}}
\newcommand{\SET}{\mathrm{SET}}
\newcommand{\excise}[1]{} %{$\star$\textsc{#1}$\star$}

\usepackage[colorinlistoftodos]{todonotes}

\usepackage[colorinlistoftodos]{todonotes}

\ytableausetup{centertableaux}

\begin{document}

%%%%%%%%%%%%%%%%%%%%%%%%%%%%%%%%%%%%%%%%%%%%%%%%%%%%%%%%%%%%
%  TITLE PAGE information
%%%%%%%%%%%%%%%%%%%%%%%%%%%%%%%%%%%%%%%%%%%%%%%%%%%%%%%%%%%%

%     [Short Title]{Full Title}

\title[Weak Bruhat interval modules of finite-type $0$-Hecke algebras]{Weak Bruhat interval modules of finite-type $0$-Hecke algebras and projective covers}
\author[J. Bardwell]{Joshua Bardwell}
\address{Department of Mathematics and Statistics, University of Otago, 730 Cumberland St., Dunedin 9016, New Zealand}
\email{barjo848@student.otago.ac.nz}

\author[D. Searles]{Dominic Searles}
\address{Department of Mathematics and Statistics, University of Otago, 730 Cumberland St., Dunedin 9016, New Zealand}
\email{dominic.searles@otago.ac.nz}


%\thanks{}

%    General info
\subjclass[2020]{Primary 05E10, 20C08, Secondary 05E05}

\date{November 16, 2023}

%\dedicatory{}

\keywords{$0$-Hecke algebra, Coxeter group, projective cover, quasisymmetric functions}

\begin{abstract}
We extend the recently-introduced weak Bruhat interval modules of the type A $0$-Hecke algebra to all finite Coxeter types. We determine, in a type-independent manner, structural properties for certain general families of these modules, with a primary focus on projective covers and injective hulls. We apply this approach to recover a number of results on type A $0$-Hecke modules in a uniform way, and obtain some additional results on recently-introduced families of type~A $0$-Hecke modules. 
\end{abstract}

\maketitle

%\tableofcontents

\section{Introduction}
The $0$-Hecke algebra $H_W(0)$ associated to a finite Coxeter group $W$ is a certain deformation of the group algebra of $W$. In \cite{Norton}, Norton classified the projective indecomposable $H_W(0)$-modules and the simple $H_W(0)$-modules up to isomorphism.  
Further structure was established by Fayers \cite{Fayers}, who proved that $H_W(0)$ is a Frobenius algebra, and introduced equivalences of the category $H_W(0)$-mod arising from three natural (anti-)automorphisms of $H_W(0)$. More recently, Huang \cite{Huang} gave a combinatorial interpretation of the projective indecomposable $H_W(0)$-modules in classical type in terms of ribbon tableaux.

The $0$-Hecke algebras in type A have attracted substantial recent interest in regard to their connection with the Hopf algebra of quasisymmetric functions. 
The quasisymmetric characteristic map \cite{DKLT} identifies the simple $0$-Hecke modules in type A with the fundamental quasisymmetric functions, which enjoy wide-ranging algebraic and combinatorial applications. There has been significant recent activity regarding constructing $0$-Hecke modules that correspond to notable bases of quasisymmetric functions, e.g., \cite{Bardwell.Searles, BBSSZ:modules, NSvWVW:modules, Searles:0Hecke, TvW:1}. 

There has also been significant work on understanding the structure of these modules. Each of \cite{Bardwell.Searles, BBSSZ:modules, NSvWVW:modules, Searles:0Hecke, TvW:1} provide a classification of indecomposability, and further work on indecomposability for the modules in \cite{TvW:1} and generalisations of these modules in \cite{TvW:2} appears in \cite{Konig} and \cite{CKNO:tableaux}.
In \cite{CKNO:projective}, Choi, Kim, Nam and Oh determined the projective covers for the modules in \cite{BBSSZ:modules, Searles:0Hecke, TvW:1, TvW:2} using the description of the type A projective indecomposable $H_W(0)$-modules in terms of ribbon tableaux due to Huang~\cite{Huang}. This technique was also employed in \cite{Kim.Yoo} to determine projective covers for further modules associated to quasisymmetric functions.

A more general family of $0$-Hecke modules in type A, called \emph{weak Bruhat interval modules}, were introduced by Jung, Kim, Lee and Oh \cite{JKLO}, for which the underlying spaces are intervals in left weak Bruhat order on the symmetric group. These modules have proven useful in providing a uniform approach to studying modules associated to families of quasisymmetric functions, and in particular their indecomposable decompositions.  
The images of weak Bruhat interval modules under compositions of the equivalences of categories in \cite{Fayers} were determined in \cite{JKLO};  
an important application stems from the fact that in certain cases, images of modules associated to one important family of quasisymmetric functions are modules for another. 
In particular, these functors were used  in \cite{JKLO} to recover and extend indecomposability results and determine injective hulls for a generalisation of the modules in \cite{Bardwell.Searles}, by realising them as images of modules in \cite{TvW:1, TvW:2}. 

In this paper, we expand on these results and techniques in a type-uniform manner. The projective indecomposable $H_W(0)$-modules play a significant role in our work; a main ingredient is a natural, type-independent realisation of these modules in terms of \emph{right descent classes}: those elements of $W$ with a specified set of right descents. First, we extend the notion of weak Bruhat interval modules to arbitrary finite Coxeter type, and show the projective indecomposable $H_W(0)$-modules are themselves weak Bruhat interval modules, which was shown for the type A case in \cite{JKLO}. 

Since the equivalences of categories in \cite{Fayers} are defined on $H_W(0)$-mod, the work of \cite{JKLO} in determining the images of weak Bruhat interval modules in type A extends naturally to arbitrary finite type. We extend this further to determine images of quotients and submodules of weak Bruhat interval modules under certain compositions of these functors, allowing applications to more general families of modules. 
We also identify a type-independent indecomposability criterion that covers a significant family of weak Bruhat interval modules, including several of the type A families of modules associated to quasisymmetric functions. 

We then determine, in a type-independent manner, the projective covers for a larger family of $H_W(0)$-modules. 
Our approach works directly with elements of the Coxeter group $W$ and left and right descents, and yields a description of the projective covers in terms of right descent classes in $W$. We then apply our result on images of quotients of weak Bruhat interval modules under the equivalences of categories to obtain the injective hulls of a corresponding family of $H_W(0)$-modules.

Finally, we specialise our attention to type A families of $0$-Hecke modules that are associated to bases of quasisymmetric functions. 
We apply the preceding results in this context to uniformly recover a number of known results on indecomposability, projective covers and injective hulls in the language of right descent sets. We additionally determine projective covers and injective hulls for certain new families of modules introduced in \cite{NSvWVW:modules}.

%%%%%%%%%%%%%%%%%%%%%%%%%%%%%%%%%%%%%%%%%%%%%%

\section{$0$-Hecke algebras for finite Coxeter systems}\label{sec:0Hecke}
A \emph{finite Coxeter system} $(W, S)$ is a finite group $W$ with generating set $S$ satisfying the relations $s^2=1$ for all $s\in S$, and $(st)_{m(s,t)}=(ts)_{m(s,t)}$ for all pairs of distinct elements $s,t\in S$, where $m(s,t)=m(t,s)\in \mathbb{Z}_{\ge 2}$ and $(st)_{m(s,t)}$ denotes the alternating product of $s$ and $t$ with $m(s,t)$ factors.   
Let $w \in W$. An expression $w = s_1\cdots s_k$ with $s_1, \ldots , s_k\in S$ is a \emph{reduced word} for $w$ if $w$ cannot be expressed as a product of elements of $S$ with fewer than $k$ terms. The \emph{length} of $w$, denoted $\ell(w)$, is the number of elements of $S$ used in any reduced word for $w$, that is, if $s_1\cdots s_k$ is a reduced word for $w$, then $\ell(w)=k$.

For each $w \in W$ and $s \in S$, either $\ell(sw) = \ell(w) - 1$ or $\ell(sw) = \ell(w) + 1$. In the former case, $s$ is a \emph{left descent} of $w$, and in the latter case, $s$ is a \emph{left ascent} of $w$. Similarly, $s$ is a \emph{right descent} of $w$ if $\ell(ws) = \ell(w)-1$, and $s$ is a \emph{right ascent} of $w$ if $\ell(ws) = \ell(w)+1$. The set of left descents of $w$ is denoted $\D_L(w)$, and the set of right descents of $w$ is denoted $\D_R(w)$. 

For $I \subseteq S$, the \emph{right descent class} $\mathcal{D}_I$ comprises the elements $w \in W$ such that $\D_R(w) = I$. 
   Let $\mathcal{D}_I^J$ denote the union of right descent classes $\mathcal{D}_X$ such that $I \subseteq X \subseteq J$, that is,
\[
\mathcal{D}_I^J = \{ w \in W :  I \subseteq \D_R(w) \subseteq J \}.
\]

The \emph{parabolic subgroup} $W_I$ is the subgroup of $W$ generated by $I$. Let $w_0(I)$ denote the longest element in $W_I$, that is, $\ell(w) < \ell(w_0(I))$ for all $w \in W_I\setminus \{w_0(I)\}$. The element $w_0(S)$ is the longest element in $W$, and is denoted by $w_0$.

Let $\K$ be a field. The \emph{0-Hecke algebra} $H_W(0)$ of a finite Coxeter system $(W,S)$ is the associative $\K$-algebra  generated by $\{\pi_s : s \in S\}$ subject to the relations 
 \begin{equation} \label{eq:coxeterrelations}
  \pi_s^2 = \pi_s  \hspace{3mm} \text{ and } \hspace{3mm} (\pi_s \pi_t)_{m(s,t)} = (\pi_t \pi_s)_{m(s,t)} 
 \end{equation}
 for all distinct $s, t \in S$. 
 
 For example, when $W$ is the symmetric group $\sym_n$ and $S$ the set $\{s_1, \ldots ,  s_{n-1}\}$ of simple transpositions, the relations \eqref{eq:coxeterrelations} are
\begin{align*}
\pi_{s_i}^2 &= \pi_{s_i}  & \text{ for } i \in [n-1],\\
\pi_{s_i}\pi_{s_j} &= \pi_{s_j} \pi_{s_i} & \text{ for }  |i - j| \geq 2, \\
\pi_{s_i}\pi_{s_{i+1}}\pi_{s_i} &= \pi_{s_{i+1}}\pi_{s_i}\pi_{s_{i+1}} & \text{ for } i \in [n-2].
\end{align*}
 
An alternative set of generators for $H_W(0)$ is given by $\{\overline{\pi}_s : s \in S\}$, where $\overline{\pi}_s = \pi_s -1$. The relations for this generating set are $\overline{\pi}_s^2 = - \overline{\pi}_s$ and $(\overline{\pi}_s \overline{\pi}_t)_{m(s,t)} = (\overline{\pi}_t \overline{\pi}_s)_{m(s,t)}$ for all distinct $s, t \in S$. Given $w\in W$ with reduced word $w = s_{1}\ldots s_{k}$, define $\pi_w$ to be the product $\pi_{s_1} \cdots \pi_{s_k}$, and define $\overline{\pi}_w$ to be $\overline{\pi}_{s_1} \cdots \overline{\pi}_{s_k}$. The following result is due to Norton \cite{Norton}.

\begin{theorem}\cite[Theorem 4.12(2)]{Norton}\label{thm:nortonresults} Let $(W,S)$ be a finite Coxeter system and let $I \subseteq S $. The left ideal $\mathcal{P}_I \coloneqq H_W(0)\pi_{w_0(I)} \overline{\pi}_{w_0(S \setminus I)}$ is a projective indecomposable $H_W(0)$-module with $\K$-basis $\{\pi_w \overline{\pi}_{w_0(S \setminus I)} : w \in \mathcal{D}_I\}$.
\end{theorem}

The set $\{\mathcal{P}_I : I \subseteq S\}$ is a complete list of non-isomorphic projective indecomposable $H_W(0)$-modules. For $I\subseteq J \subseteq S$, let $\mathcal{P}_I^J$ denote the $H_W(0)$-module $H_W(0)\pi_{w_0(I)}\overline{\pi}_{w_0(S \setminus J)}$. 

The following result is entirely analogous to that of Huang in \cite[Theorem 3.2]{Huang};  for our purposes, it is more convenient to work with generators $\pi_s$ rather than $\overline{\pi}_s$, and assign different roles to the indexing sets $I$ and $J$. 

\begin{theorem}
    Let $I \subseteq J \subseteq S$. Then $\mathcal{P}_I^J$ has a $\K$-basis
    \begin{align}
        \{\pi_w \overline{\pi}_{w_0(S \setminus J)} : w \in W \text{ and } I \subseteq \textup{D}_R(w) \subseteq J\}, \label{eq:projbasis}
    \end{align}
    and decomposes as a direct sum of projective indecomposable modules via the formula
\begin{align}
    \mathcal{P}_I^J \cong \bigoplus_{I \subseteq X \subseteq J} \mathcal{P}_X. \label{eq:altdecomp}
\end{align}
\end{theorem}
 \begin{proof}
     Let $w \in W$ satisfy $I \subseteq \D_R(w) \subseteq J$. Using the fact that $\pi_s\overline{\pi}_s = 0$ for all $s \in S$, it is straightforward to establish that
 \begin{align}
   \pi_s  (\pi_w\overline{\pi}_{w_0(S \setminus J)}) =
  \begin{cases}
                                   \pi_w \overline{\pi}_{w_0(S \setminus J)} & \text{if $s \in \D_L(w)$,} \\     
                                   \pi_{sw} \overline{\pi}_{w_0(S \setminus J)} & \text{if $s \notin \D_L(w)$ and $\D_R(sw) \subseteq J $,} \\ 
                                   0 & \text{if $s \notin \D_L(w)$ and $\D_R(sw) \nsubseteq J $,}                
  \end{cases} \label{eq:action5}
  \end{align}
for all $s \in S$, and therefore \eqref{eq:projbasis} is a $\K$-basis for $\mathcal{P}_I^J$. 

For the decomposition, let $X_1, \ldots, X_\ell$ denote the subsets of $S$ such that $I \subseteq X_i \subseteq J$ for all $i \in [\ell]$, ordered via any total ordering satisfying $i\le j$ whenever $X_i \subseteq X_j$. Let $M_i$ denote the $H_W(0)$-module defined as the $\K$-span of the disjoint union 
  \begin{equation}\nonumber
      \bigcup_{j \geq i}  \mspace{1mu} 
      \left\{\pi_w \overline{\pi}_{w_0(S \setminus J)} : w \in \mathcal{D}_{X_j}\right\} 
  \end{equation}
  and consider the filtration $\mathcal{P}_I^J = M_1 \supseteq M_2 \ldots \supseteq M_\ell \supseteq M_{\ell+1} = 0$. Each quotient  $M_i / M_{i+1}$ has basis $\{\pi_w \overline{\pi}_{w_0(S \setminus J)} : w \in \mathcal{D}_{X_{i}}\}$. Thus from \eqref{eq:action5} one has $M_i / M_{i+1} \cong \mathcal{P}_{X_i}$ and the isomorphism \eqref{eq:altdecomp} follows.
 \end{proof}

%%%%%%%%%%%%%%%%%%%%%%%%%%%%%%%%%%%%%%%%%%%%%%%%%%%%%%%%%

\section{Weak Bruhat interval modules}\label{sec:bruhat}

In this section, we extend the type A weak Bruhat interval modules of Jung, Kim, Lee and Oh \cite{JKLO} to arbitrary finite type. We identify an indecomposability criterion covering an important family of weak Bruhat interval modules, and extend results in \cite{JKLO} concerning functors on the category $H_W(0)$-mod to submodules and quotients of weak Bruhat interval modules in finite type. 

The \emph{left weak Bruhat order $\leq_L$ on $W$} is the partial order defined by $u\le_L v$ if there exist some $s_1, \ldots , s_k\in S$ such that $v = s_1 \cdots s_k u$ and $\ell(v) = \ell(u)+k$. Given $u, v\in W$ with $u\le_L v$, the \emph{left weak Bruhat interval} is the set $[u,v]_L = \{w\in W : u\le_L w \le_L v\}$.

\begin{definition}\label{def:intervalmodule} Let $[u,v]_L \subseteq W$. The \emph{weak Bruhat interval module} $\B(u,v)$ is the $H_W(0)$-module $\K[u,v]_L$ equipped with the $H_W(0)$-action defined by
\begin{align}
  \pi_s  w =
  \begin{cases}
                                   w & \text{if $s \in \D_L(w)$,} \\      
                                   s w & \text{if $s \notin \D_L(w) $ and $sw \in [u,v]_L$,} \\
                                   0 & \text{if $s \notin \D_L(w) $ and $sw \notin [u,v]_L$,} 
  \end{cases} \label{eq:bruhataction}
\end{align}

 for all $s \in S$ and $w\in [u,v]_L$. 
\end{definition}

That \eqref{eq:bruhataction} defines an action of $H_W(0)$ follows from Theorems 3.1 and 3.3 in \cite{Defant.Searles}, noting that the ascent-compatibility condition in \cite{Defant.Searles} does not depend on whether a negative sign appears when elements are fixed by $\pi_s$. Note that the type A case of Definition~\ref{def:intervalmodule} is precisely \cite[Definition 3.1(1)]{JKLO}.

\begin{example}  
Let $W=\sym_4$; we write elements of symmetric groups in one-line notation. Figure~\ref{fig:digraph} depicts the action of $\pi_1, \pi_2$ and $\pi_3$ (where $\pi_i$ denotes $\pi_{s_i}$) on the basis $[1324, 1432]_L$ of $\B(1324, 1432)$ and on the basis $[3142, 4231]_L$ of $\B(3142, 4231)$. Following the convention in \cite{JKLO}, we draw Hasse diagrams from top to bottom, rather than bottom to top, and so the $0$-Hecke operators move elements downwards (or send them to zero). 
\end{example}

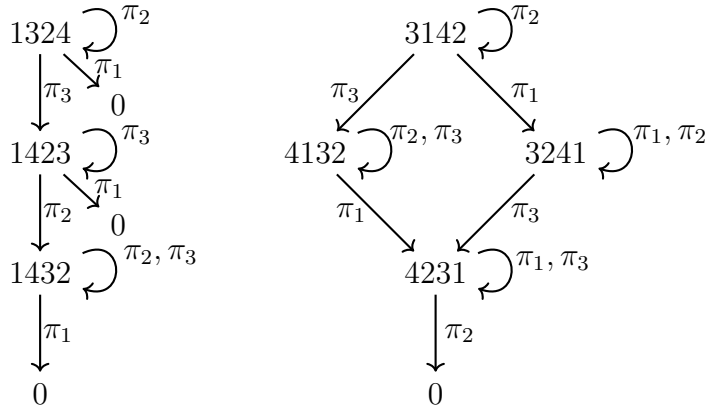
\begin{figure}[h]
\centering
\begin{tikzpicture}[thick, scale=0.8]
  
  \node (1) at (0,2) {$1432$};
  \node (2) at (0,4) {$1423$};
  \node (3) at (0,6) {$1324$};
  \node (30) at (1.3,4.8) {$0$};
  \node (20) at (1.3,2.8) {$0$};
  \node (10) at (0,0) {$0$};
  
 \node (s1) at (0.3,3) {$\pi_2$};
 \node (s2) at (0.3,5) {$\pi_3$};
 
  \node (p1) at (2,2.3) {$\pi_2, \pi_3$};
  \node (p2) at (1.6,4.3) {$\pi_3$};
  \node (p3) at (1.6,6.3) {$\pi_2$};
  
  \node (p1*) at (0.3,1) {$\pi_1$};
  \node (p2*) at (1.15,3.4) {$\pi_1$};
  \node (p3*) at (1.15,5.4) {$\pi_1$};
  
\draw[->] (3)--(2);
\draw[->] (2)--(1);

\draw[->] (1)--(10);
\draw[->] (2)--(20);
\draw[->] (3)--(30);

\draw [->] (1) to [out=25,in=340, looseness=3.5] (1);
\draw [->] (2) to [out=25,in=340, looseness=3.5] (2);
\draw [->] (3) to [out=25,in=340, looseness=3.5] (3);

\end{tikzpicture} 
\hspace{6mm}
\begin{tikzpicture}[thick, scale=0.8]
  
  \node (1) at (2, 2) {$4231$};
  \node (2) at (0, 4) {$4132$};
  \node (2') at (4, 4) {$3241$};
  \node (3) at (2, 6) {$3142$};
  \node (10) at (2, 0) {$0$};
  
 \node (s1) at (3.5, 3) {$\pi_3$};
 \node (s2) at (3.5, 5) {$\pi_1$};
 
  \node (p1) at (0.6, 3) {$\pi_1$};
  \node (p2) at (1.85, 4.3) {$\pi_2, \pi_3$};
  \node (p22) at (5.9, 4.3) {$\pi_1, \pi_2$};
  \node (p3) at (3.5, 6.3) {$\pi_2$};
  
  \node (p1*) at (2.4, 1) {$\pi_2$};
  \node (p1**) at (3.95, 2.2) {$\pi_1, \pi_3$};
 % \node (p2*) at (1.15, 3.3) {$\pi_1$};
  \node (p3*) at (0.5, 5) {$\pi_3$};
  
\draw[->] (3)--(2);
\draw[->] (2)--(1);
\draw[->] (2')--(1);
\draw[->] (3)--(2');

\draw[->] (1)--(10);

\draw [->] (1) to [out=25,in=340, looseness=3.5] (1);
\draw [->] (2) to [out=25,in=340, looseness=3.5] (2);
\draw [->] (3) to [out=25,in=340, looseness=3.5] (3);
\draw [->] (2') to [out=25,in=340, looseness=3.5] (2');

\end{tikzpicture} 
\caption{The $H_{\sym_4}(0)$-action on $\K$-bases for $\B(1324, 1432)$ and $\B(3142, 4231)$.}
\label{fig:digraph}
\end{figure}

We next obtain a natural interpretation of the $H_W(0)$-modules $\mathcal{P}_I^J$, and thus the projective indecomposable $H_W(0)$-modules $\mathcal{P}_I$, as weak Bruhat interval modules. Given $I\subseteq J \subseteq S$, the union $\mathcal{D}_I^J$ of right descent classes is an interval in left weak Bruhat order \cite[Theorem 6.2]{Bjorner.Wachs}. In particular, each right descent class $\mathcal{D}_I$ itself is an interval in left weak Bruhat order.

\begin{example}
Figure~\ref{fig:S4} shows the poset $(\sym_4, \le_L)$ with each $\mathcal{D}_I$ coloured separately, aside from the two single-element classes which are uncoloured.
\end{example}

\begin{figure}[h]
\centering
\begin{tikzpicture}[thick, scale=0.75]
\tikzstyle{every node} = [draw, circle,scale = 0.3]

  \node (1234) at (0,0.5) {$$};
  
  \node (1243) at (-1.5,1.8) {$$};
  \node (1324) at (0,1.8) {$$};
  \node (2134) at (1.5,1.8) {$$};
  
  \node (1342) at (-3,3.1) {$$};
  \node (1423) at (-1.5,3.1) {$$};
  \node (2143) at (0,3.1) {$$};
  \node (2314) at (1.5,3.1) {$$};
  \node (3124) at (3,3.1) {$$};
  
  \node (1432) at (-3.8,4.4) {$$};
  \node (2341) at (-2.23,4.4) {$$};
  \node (3142) at (-0.78,4.4) {$$};
  \node (2413) at (0.75,4.4) {$$};
  \node (3214) at (2.25,4.4) {$$};
  \node (4123) at (3.8,4.4) {$$};
  
  \node (2431) at (-3,5.7) {$$};
  \node (3241) at (-1.5,5.7) {$$};
  \node (3412) at (0,5.7) {$$};
  \node (4132) at (1.5,5.7) {$$};
  \node (4213) at (3,5.7) {$$};
  
  \node (3421) at (-1.5,7) {$$};
  \node (4231) at (0,7) {$$};
  \node (4312) at (1.5,7) {$$};
  
  \node (4321) at (0,8.2) {$$};
  
  \tikzstyle{every node} = [draw, circle,scale = 1]
  
  \draw[dashed, gray] (3421) -- (4321);
  \draw[dashed, gray] (4231) -- (4321);
  \draw[dashed, gray] (4312) -- (4321);
  \draw[dashed, gray] (1432) -- (1342); 
  \draw[dashed, gray] (2341) -- (2431); 
  \draw[dashed, gray] (2341) -- (3241); 
  \draw[dashed, gray] (3421) -- (3421); 
  \draw[dashed, gray] (4312) -- (3412); 
  \draw[dashed, gray] (1234) -- (2134); 
  \draw[dashed, gray] (3412) -- (3421); 
  \draw[dashed, gray] (1243) -- (2143); 
  \draw[dashed, gray] (2134) -- (2143); 
  \draw[dashed, gray] (3214) -- (3124);   
  \draw[dashed, gray] (4213) -- (4123);   
  \draw[dashed, gray] (2314) -- (3214);   
  \draw[dashed, gray] (1423) -- (1432);   
  \draw[dashed, gray] (4123) -- (4132);   
  \draw[dashed, gray](1234) -- (1243);
  \draw[dashed, gray](1234) -- (1324);  

  \draw [Purple] (1243) -- (1342) -- (2341);
  \draw [orange] (1432) -- (2431) -- (3421);
  \draw [Maroon] (3214) -- (4213) -- (4312);
  \draw [Turquoise] (2134) -- (3124) -- (4123);
  \draw [WildStrawberry] (1324) -- (2314) -- (2413) -- (3412);
  \draw [WildStrawberry] (1324) -- (1423) -- (2413);
  \draw[ForestGreen] (4231) -- (3241) -- (3142);
  \draw[ForestGreen] (4231) -- (4132) -- (3142) -- (2143);
  
 \draw (4321) node[draw=none,above, scale=0.8] {$1234$};
 
 %%%%%%%%%%%%%%%%%%%%%%%%%%%%%%%
 
\draw (3421) node[draw=none,above left, scale=0.8] {$2134$};
\draw (4231) node[draw=none, above left, scale=0.8] {$1324$}; 
\draw (4312) node[draw=none,above right, scale=0.8] {$1243$};

%%%%%%%%%%%%%%%%%%%%%%%%%%%%%%%%

\draw (2431) node[draw=none,above left, scale=0.8] {$3124$};
\draw (3241) node[draw=none,above left, scale=0.8] {$2314$};
\draw (3412) node[draw=none,above, scale=0.8] {$2143$};
\draw (4132) node[draw=none,above right, scale=0.8] {$1423$};
\draw (4213) node[draw=none,above right, scale=0.8] {$1342$};

%%%%%%%%%%%%%%%%%%%%%%%%%%%%%%%%
 
\draw (1432) node[draw=none,left, scale=0.8] {$4123$};
\draw (2341) node[draw=none,left, scale=0.8] {$3214$};
\draw (3142) node[draw=none,left, scale=0.8] {$2413$};
\draw (2413) node[draw=none,right, scale=0.8] {$3142$};
\draw (3214) node[draw=none,right, scale=0.8] {$2341$}; 
\draw (4123) node[draw=none,right, scale=0.8] {$1432$};

%%%%%%%%%%%%%%%%%%%%%%%%%%%%%%%%
 
\draw (1342) node[draw=none,below left, scale=0.8] {$4213$};
\draw (1423) node[draw=none,below left, scale=0.8] {$4132$};
\draw (2143) node[draw=none,below, scale=0.8] {$3412$};
\draw (2314) node[draw=none,below right, scale=0.8] {$3241$};  
\draw (3124) node[draw=none,below right, scale=0.8] {$2431$};

%%%%%%%%%%%%%%%%%%%%%%%%%%%%%%%%

 \draw (1243) node[draw=none,below, scale=0.8] {$4312$};
 \draw (1324) node[draw=none,below right , scale=0.8] {$4231$};
 \draw (2134) node[draw=none, below right, scale=0.8] {$3421$};

%%%%%%%%%%%%%%%%%%%%%%%%%%%%%%%% 
 
\draw (1234) node[draw=none,below, scale=0.8] {$4321$}; 
 
\end{tikzpicture}
\vspace{-5mm}
\caption{The poset $(\sym_4, \leq_L)$ and the right descent classes $\mathcal{D}_I$.}
\label{fig:S4}
\end{figure}
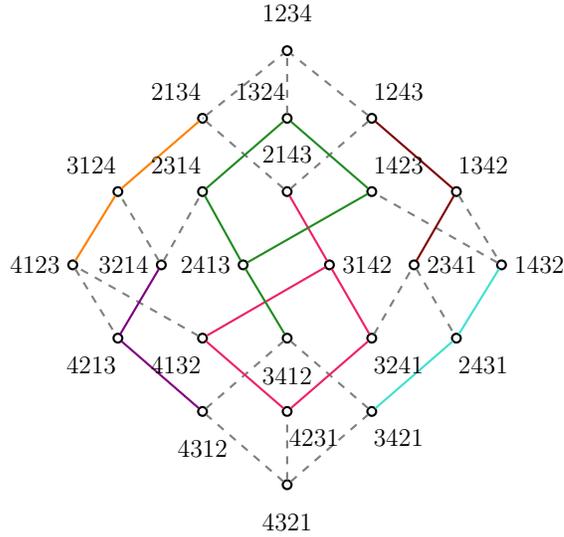

The shortest element in $\mathcal{D}_I$ is $w_0(I)$, and the longest element in $\mathcal{D}_I$ is $w_0w_0(S\setminus I)$. 
This notation is potentially confusing due to the conflict between $w_0$ as the longest element in $W$ and $w_0$ as the function returning the longest element in $W$ with given right descents, and becomes especially cumbersome when we multiply or conjugate these elements by $w_0$. Therefore, for the remainder of the paper, we denote the shortest element in $\mathcal{D}_I$ by $u_I$ and the longest element in $\mathcal{D}_I$ by $v_I$.

\begin{theorem}\label{thm:modified} Let $I \subseteq J \subseteq S$. Then $\mathcal{P}_I^J \cong \B(u_I,v_J)$.
\end{theorem}
\begin{proof}
The weak Bruhat interval $[u_I,v_J]_L$ is precisely $\mathcal{D}_I^J$, and, for any $w \in \mathcal{D}_I^J$ and $s \notin \D_L(w)$, $\D_R(sw) \nsubseteq J$ is equivalent to $sw \notin \mathcal{D}_I^J$. Hence the action \eqref{eq:bruhataction} on $\B(u_I,v_J)$ is the action \eqref{eq:action5} on $\mathcal{P}_I^J$ when identifying $w$ with $\pi_w \overline{\pi}_{w_0(S \setminus J)}$. 
\end{proof}

To emphasise their nature as (direct sums of) projective indecomposable $H_W(0)$-modules, for the remainder of the paper we denote the weak Bruhat interval module $\B(u_I, v_J)$ by $\PP_I^J$, and we denote the projective indecomposable weak Bruhat interval module $\B(u_I, v_I)$ by $\PP_I$.

 \begin{example} Consider the $H_{\sym_4}(0)$-module $\B(2134, 4231)$, and let $i$ denote $s_i$. Since $2134 = u_{\{1\}}$ and $4231 = v_{\{1,3\}}$, by Theorem \ref{thm:modified} we have $\mathcal{P}_{\{1\}}^{\{1,3\}} \cong \B(2134, 4231) = \PP_{\{1\}}^{\{1,3\}} \cong \PP_{\{1\}} \oplus \PP_{\{1,3\}}$. The basis elements for $\PP_{\{1\}}^{\{1,3\}}$ are depicted in Figure~\ref{fig:doubledescent} with those belonging to $\PP_{\{1\}}$ in orange and those belonging to $\PP_{\{1,3\}}$ in pink (cf. Figure~\ref{fig:S4}). 
 \end{example}

\begin{figure}[h]
\centering
\begin{tikzpicture}[thick, scale=0.65]
\tikzstyle{every node} = [draw, circle,scale = 0.3]
  
  \node (1324) at (0,1.8) {$$};
  \node (1423) at (-1.5,3.1) {$$};
  \node (2314) at (1.5,3.1) {$$};
  \node (1432) at (-3.8,4.4) {$$};
  \node (2413) at (0.75,4.4) {$$};
  \node (2431) at (-3,5.7) {$$};
  \node (3412) at (0,5.7) {$$};
  \node (3421) at (-1.5,7) {$$};

  \tikzstyle{every node} = [draw, circle,scale = 1]

  \draw[dashed, gray] (3421) -- (3421); 
  
  \draw[dashed, gray] (3412) -- (3421); 
  \draw[dashed, gray] (1423) -- (1432);   

  \draw [Orange] (1432) -- (2431) -- (3421);
  \draw [WildStrawberry] (1324) -- (2314) -- (2413);
  \draw [WildStrawberry] (1324) -- (1423);
  \draw [WildStrawberry] (1324) -- (1423);
  \draw [WildStrawberry] (1423) -- (2413) -- (3412);

 \draw (1432) node[draw=none,left, scale=0.8] {$4123$};
   
 \draw (2431) node[draw=none,above left, scale=0.8] {$3124$};
  
 \draw (3421) node[draw=none,above left, scale=0.8] {$2134$};
       
 \draw (1423) node[draw=none, below left, scale=0.8] {$4132$};
 
 \draw (2413) node[draw=none,right, scale=0.8] {$3142$};
 
 \draw (3412) node[draw=none,right, scale=0.8] {$2143$};
 
  \draw (1324) node[draw=none,left, scale=0.8] {$4231$};
  
  \draw (2314) node[draw=none,right, scale=0.8] {$3241$};
\end{tikzpicture}

\caption{The basis elements for $\B(2134, 4231) \cong \PP_{\{1\}} \oplus \PP_{\{1,3\}}$.} 
\label{fig:doubledescent}
\end{figure}
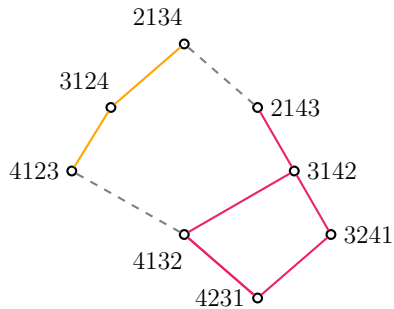

The \emph{socle} of a module $M$ is the sum of all simple submodules of $M$, that is, the largest semisimple submodule of $M$, and the \emph{top} of $M$ is the largest semisimple quotient of $M$.  
The following indecomposability criterion follows naturally from the algebraic structure of $H_W(0)$.

\begin{proposition}\label{prop:indecomposable}
Every submodule and quotient of $\PP_I$ is indecomposable. 
\end{proposition}
\begin{proof}
It was shown in \cite{Fayers} that $H_W(0)$ is a Frobenius algebra, therefore self-injective, hence the projective and injective $H_W(0)$-modules coincide.  
It then follows from \cite[Proposition 4.1(d)]{Auslander} that the socle of any projective indecomposable $H_W(0)$-module $\PP_I$ is simple. Thus every submodule of $\PP_I$ contains the socle of $\PP_I$, and so any two nonzero submodules of $\PP_I$ have nontrivial intersection. Therefore no submodule of $\PP_I$ can be a direct sum of nonzero submodules.

For quotients, if there is a surjection from $\PP_I$ to a direct sum of two nonzero modules, then, choosing a simple quotient of each summand, one has a surjection from $\PP_I$ to a direct sum of two simple modules. However, since the top of $\PP_I$ is simple (\cite{Norton}), this is a contradiction.
\end{proof}

The socle and the top of $\PP_I$ can be identified explicitly in terms of weak Bruhat interval modules: the socle of $\PP_I$ is $\B(v_I, v_I)$, whereas the top of $\PP_I$ is $\B(u_I, u_I)$.

\begin{lemma}\label{lem:quotientinterval} Let $u,v \in W$ and $Y \subseteq [u, v]_L$. Then $\K Y$ is a $H_W(0)$-submodule of $\B(u, v)$ if and only if $Y$ is an upper order ideal in the poset $[u, v]_L$. Moreover, if $w \in [u, v]_L$, then $[u, v]_L \setminus [u, w]_L$ is an upper order ideal in $[u, v]_L$.
\end{lemma}
\begin{proof} 
The first statement is immediate from the definition of $\B(u,v)$. For the second, suppose there exists some $x \in [u,  v]_L \setminus [u, w]_L$ and some $s\in S$ such that $\ell(sx)>\ell(x)$, $sx\in [u,v]_L$, but $sx \notin [u, v]_L \setminus [u, w]_L$. Then $sx \in [u, w]_L$, so $u \le_L x <_L sx \le_L w $. Therefore $x \in [u, w]_L$, contradicting the assumption that $x \in [u, v]_L \setminus [u, w]_L$. 
\end{proof}

The following corollary specialises Proposition~\ref{prop:indecomposable} to weak Bruhat interval modules.

\begin{proposition}\label{prop:indecomposablebruhat}
The weak Bruhat interval modules $\B(w, v_I)$ and $\B(u_I, w)$ are indecomposable for all $w \in \mathcal{D}_I$. Moreover, any submodule of $\B(w, v_I)$ and any quotient of $\B(u_I, w)$ is also indecomposable. 
\end{proposition}
\begin{proof}
By Lemma~\ref{lem:quotientinterval}, we have that $\B(u_I, w)$ is a quotient of $\PP_I$. The statement then follows immediately from Proposition~\ref{prop:indecomposable}.
\end{proof}

Several families of $0$-Hecke modules associated to quasisymmetric functions are isomorphic to weak Bruhat interval modules that are either submodules or quotient modules of some $\PP_I$. Proposition~\ref{prop:indecomposablebruhat} will be applied in Section~\ref{sec:applications}.

\begin{remark} \label{rmk:notsimple}
Quotients of $\B(w, v_I)$ and submodules of $\B(u_I, w)$ are not indecomposable in general. For example, consider $W=\sym_4$ and the elements $2143 = u_{\{1,3\}}$ and $4132$ in $\mathcal{D}_{\{1,3\}}$. In the module $\B(2143, 4132)$, the submodules $\K\{3142 - 4132\}$ and $\K\{4132\}$ are both simple, hence the socle of $\B(2143, 4132)$ is decomposable.
\end{remark}

We now consider functors on the category $H_W(0)$-mod introduced in \cite{Fayers} and studied in terms of type A weak Bruhat interval modules in \cite{JKLO}. We will determine images of submodules and quotients of finite-type weak Bruhat interval modules under certain compositions of these functors. We largely follow the notation of \cite{JKLO}.

Let $w^{w_0}$ denote the conjugation $w_0ww_0$. Fayers \cite{Fayers} considers involutions $\upphi$, $\uptheta$ and an anti-involution $\upchi$ on $H_W(0)$ defined by
\begin{equation*}
\upphi : \pi_s \mapsto \pi_{s^{w_0}}, \hspace{10mm}  \uptheta : \pi_s \mapsto  1 - \pi_s, \hspace{10mm} \upchi : \pi_s \mapsto  \pi_s.
\end{equation*}

Given a $H_W(0)$-module $M$, Fayers \cite{Fayers} defines $H_W(0)$-modules $\upphi[M]$, $\uptheta[M]$ and $\upchi[M]$. For $\upphi[M]$ and $\uptheta[M]$, the underlying space is $M$, and the actions $\cdot_\upphi$ and $\cdot_\uptheta$ are defined by $\pi_s \cdot_\upphi m = \upphi(\pi_s)\cdot m$ and $\pi_s \cdot_\uptheta m = \uptheta(\pi_s)\cdot m$, for $m\in M$. For $\upchi[M]$, the underlying space is the dual space $M^*$ of $M$, and the action is given by $(\pi_s \cdot^\upchi f)(m) = f(\upchi(\pi_s)\cdot m)$, for $f\in M^*$ and $m\in M$. The functors $M\mapsto \upphi[M]$ and $M\mapsto \uptheta[M]$ are self-equivalences of $H_W(0)$-mod, and the functor $M\mapsto \upchi[M]$ is a dual equivalence of $H_W(0)$-mod. 

Fayers \cite{Fayers} determined the images of the simple $H_W(0)$-modules under these functors, and Huang \cite{Huang} determined the images of the projective indecomposable $H_W(0)$-modules under $\upphi$ and $\uptheta$. In type A, Jung, Kim, Lee and Oh \cite{JKLO} determined the images of weak Bruhat interval modules under $\upphi$, $\uptheta$ and $\upchi$ and their compositions;  those important for our purposes are the involution $\upphi$ and the anti-involutions $\hat{\uptheta} \coloneqq \uptheta \circ \upchi$ and $\hat{\upomega} \coloneqq \upphi \circ \uptheta \circ \upchi$. We now extend this result on $\upphi$, $\hat{\uptheta}$, and $\hat{\upomega}$ to arbitrary finite type, and moreover to quotients and submodules of weak Bruhat interval modules defined by upper order ideals in intervals in weak Bruhat order. For $Y\subseteq W$, let $w_0Yw_0$ denote the set $\{w_0yw_0 : y\in Y\}$, and similarly for $Yw_0$ and $w_0Y$.
 
 \begin{theorem}\label{thm:autotwists} 
 Let $Y$ be an upper order ideal in $[u, v]_L$. Then
\begin{align}
&\upphi[\B(u, v)/\K Y] \cong \K([u^{w_0}, v^{w_0}]_L\setminus w_0Yw_0), \label{eq:phi} \\  & \hspace{0.3mm} \hat{\uptheta} \hspace{0.4mm} [\B(u, v)/ \K Y] \cong  \K([vw_0, u{w_0}]_L\setminus Yw_0), \label{eq:theta} \\[0.1em]  
& \hat{\upomega}[\B(u,v)/\K Y] \hspace{-0.35mm} \cong  \K([w_0v, w_0u]_L\setminus w_0Y). \label{eq:omega}
\end{align}
 \end{theorem}
\begin{proof}  We only prove \eqref{eq:omega}, since the proofs of \eqref{eq:phi} and \eqref{eq:theta} are similar. Let $Y$ be an upper order ideal in $[u,v]_L$. For any $w$ in the basis $[u, v]_L \setminus Y$ of the quotient module $\B(u,v)/\K Y$, let $w^*$ denote the dual of $w$ with respect to this basis.

The $H_W(0)$-action on $\hat{\upomega}[\B(u,v)/\K Y]$ is given by
 \begin{align}
\pi_{s} \cdot^{\hat{\upomega}} w^* = 
\begin{cases}
	w^* & \text{if $s^{w_0} \notin \D_L(w)$}, \\
	-(s^{w_0}w)^* & \text{if $s^{w_0} \in \D_L(w)$ and $s^{w_0} w \in [u,v]_L \setminus Y$}, \\
 0 &  \text{if $s^{w_0} \in \D_L(w)$ and $s^{w_0} w \notin [u,v]_L \setminus Y$},
\end{cases} \label{eq:twistaction}
\end{align}
for all $w \in [u,v]_L\setminus Y$ and $s \in S$. The map $f : \hat{\upomega}[\B(u, v)/\K Y] \to \  \K[w_0 v, w_0 u]_L \setminus w_0 Y$ defined by $f(w^*) = (-1)^{\ell(ww_0u^{-1})}w_0w$ is a bijection. 
    To show $f$ is an isomorphism, we compute
\begin{align*}
f(\pi_{s} \cdot ^{\hat{\upomega}} w^*) = 
\begin{cases}
	(-1)^{\ell(ww_0u^{-1})}w_0w & \text{if $s^{w_0} \notin \D_L(w)$}, \\
	(-1)^{\ell(s^{w_0}ww_0u^{-1})+1}sw_0w & \text{if $s^{w_0} \in \D_L(w)$ and $s^{w_0}w \in [u,v]_L \setminus Y$}, \\
 0 &  \text{if $s^{w_0} \in \D_L(w)$ and $s^{w_0} w \notin [u,v]_L \setminus Y$,}
\end{cases} 
\end{align*}
and 
\begin{align*}
\pi_{s} f(w^*) = 
\begin{cases}
	(-1)^{\ell(ww_0u^{-1})}w_0w & \text{if $s \in \D_L(w_0w)$}, \\
	(-1)^{\ell(ww_0u^{-1})}sw_0w & \text{if $s \notin \D_L(w_0w)$ and $sw_0w  \in [w_0v,w_0u]_L \setminus w_0Y$}, \\
 0 &  \text{if $s \notin \D_L(w_0 w)$ and $sw_0w \notin [w_0v, w_0 u]_L \setminus w_0Y$.}
\end{cases} 
\end{align*}

Then $f(\pi_s \cdot^{\hat{\upomega}} w^*) = \pi_s  f(w^*)$ follows from that fact that $s \in \D_L(w_0w)$ if and only if $s^{w_0} \notin \D_L(w)$, that $s^{w_0} w \in [u,v]_L \setminus Y$ if and only if $sw_0w \notin [w_0v, w_0u]_L \setminus w_0 Y$, and that  $s^{w_0}\in \D_L(w)$ implies $\ell(s^{w_0}ww_0u^{-1}) = \ell(ww_0u^{-1}) - 1$. 
\end{proof}

\begin{remark} 
Theorem~\ref{thm:autotwists} extends three of the cases in \cite[Table 3.1]{JKLO}. The remaining cases can be extended similarly. To do so requires introducing \emph{negative weak Bruhat interval modules} $\overline{\B}(u,v)$ in arbitrary finite type, analogously to the type A definition given in \cite[Definition 3.1(2)]{JKLO}. Similarly to $\B(u,v)$, well-definedness of $\overline{\B}(u,v)$ in finite type follows from \cite{Defant.Searles}. 
\end{remark} 

The following lemma, due to \cite[Proposition 2.3.4]{Bjorner1} and \cite[Example 2.10]{Bjorner1}, summarises the effect of multiplication or conjugation by $w_0$ on the shortest and longest elements of right descent classes.

\begin{lemma}\label{lem:shortestlongest} Let $I \subseteq J\subseteq S$. Then 
\begin{enumerate}
\item $u_I^{w_0} = u_{w_0Iw_0}$ and $v_I^{w_0} = v_{w_0Iw_0}$,
\item $u_Iw_0 = v_{S \setminus w_0Iw_0}$ and $v_I{w_0} = u_{S\setminus w_0Iw_0}$,
\item $w_0 u_I = v_{S \setminus I}$ and ${w_0}  v_I= u_{S \setminus I}$.
\end{enumerate}
\end{lemma}

Theorem~\ref{thm:autotwists} and Lemma~\ref{lem:shortestlongest} yield the following extension of \cite[Proposition 5.1]{Huang}, which will be applied in Sections~\ref{sec:projectivecovers} and \ref{sec:applications}.

\begin{corollary}\label{cor:PIJtwist} Let $I \subseteq J \subseteq S$. Then 
\begin{align*} \upphi[\PP_I^J] \cong \PP_{w_0I w_0}^{w_0J w_0} \,\, , \hspace{4mm} \hat{\uptheta}[\PP_I^J] \cong \PP^{S\setminus w_0I w_0}_{S\setminus w_0J w_0} \hspace{4mm} \text{and} \hspace{4mm} \hat{\upomega}[\PP_I^J] \cong \PP^{S\setminus I}_{S \setminus J}.
\end{align*} 
\end{corollary}

\begin{example}
Let $W = \sym_4$ and $I = \{1\}$, and let $i$ denote $s_i$. Then $\PP_{S \setminus I} = \PP_{\{2,3\}}$, $\PP_{w_0 I w_0} = \PP_{\{3\}}$, and $\PP_{S\setminus w_0 I w_0} = \PP_{\{1,2\}}$. Figure~\ref{fig:PIJtwists} depicts how these four projective indecomposable modules are related via Corollary~\ref{cor:PIJtwist}.
\end{example}

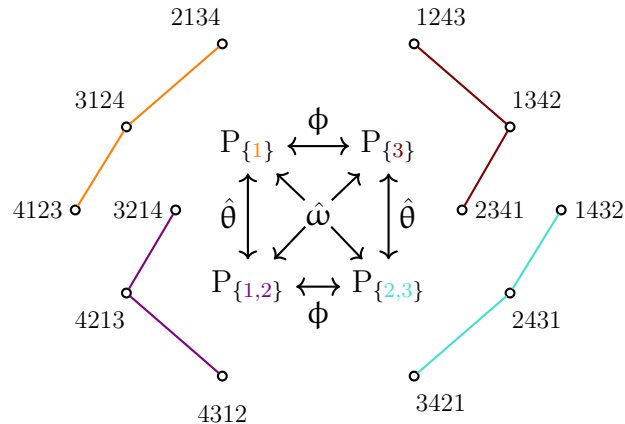
\begin{figure}[h]
\centering
\begin{tikzpicture}[thick, scale=0.85]

\node (D1) at (-1.1,5.4) {$\PP_{\{\textcolor{orange}{1}\}}$};
\node (D3) at (1.1,5.4) {$\PP_{\{\textcolor{Maroon}{3}\}}$};
\node (D12) at (-1.1,3.2) {$\PP_{\{\textcolor{Purple}{1,2}\}}$};
\node (D23) at (1.1,3.2) {$\PP_{\{\textcolor{Turquoise}{2,3}\}}$};

\node (psi) at (0,5.8) {$\upphi$};
\node (psi) at (0,2.75) {$\upphi$};
\node (rho) at (-1.4,4.35) {$\hat{\uptheta}$};
\node (rho) at (1.4,4.35) {$\hat{\uptheta}$};

\tikzstyle{every node} = [draw, circle,scale = 0.3]
  
  \node (1243) at (-1.5,1.8) {$$};
  \node (2134) at (1.5,1.8) {$$};
  
  \node (1342) at (-3,3.1) {$$};
   \node (3124) at (3,3.1) {$$};
  
  \node (1432) at (-3.8,4.4) {$$};
  \node (2341) at (-2.23,4.4) {$$};
  \node (3214) at (2.25,4.4) {$$};
  \node (4123) at (3.8,4.4) {$$};
  
  \node (2431) at (-3,5.7) {$$};
  \node (4213) at (3,5.7) {$$};
  
  \node (3421) at (-1.5,7) {$$};
  \node (4312) at (1.5,7) {$$};
  
  \tikzstyle{every node} = [draw, circle,scale = 1]

  \draw [Purple] (1243) -- (1342) -- (2341);
  \draw [orange] (1432) -- (2431) -- (3421);
  \draw [Maroon] (3214) -- (4213) -- (4312);
  \draw [Turquoise] (2134) -- (3124) -- (4123);
  
%%%%%%%%%%%%%%%%%%%%%%%%%%%%%%%
 
\draw (3421) node[draw=none,above left, scale=0.8] {$2134$};
\draw (4312) node[draw=none,above right, scale=0.8] {$1243$};

%%%%%%%%%%%%%%%%%%%%%%%%%%%%%%%%

\draw (2431) node[draw=none,above left, scale=0.8] {$3124$};
\draw (4213) node[draw=none,above right, scale=0.8] {$1342$};

%%%%%%%%%%%%%%%%%%%%%%%%%%%%%%%%
 
\draw (1432) node[draw=none,left, scale=0.8] {$4123$};
\draw (2341) node[draw=none,left, scale=0.8] {$3214$};
\draw (3214) node[draw=none,right, scale=0.8] {$2341$}; 
\draw (4123) node[draw=none,right, scale=0.8] {$1432$};

%%%%%%%%%%%%%%%%%%%%%%%%%%%%%%%%
 
\draw (1342) node[draw=none,below left, scale=0.8] {$4213$};
\draw (3124) node[draw=none,below right, scale=0.8] {$2431$};

%%%%%%%%%%%%%%%%%%%%%%%%%%%%%%%%

 \draw (1243) node[draw=none,below, scale=0.8] {$4312$};
 \draw (2134) node[draw=none, below right, scale=0.8] {$3421$};

%%%%%%%%%%%%%%%%%%%%%%%%%%%%%%%% 

\draw [<->] (D1) -- (D3);
\draw [<->] (D1) -- (D12);
\draw [<->] (D12) -- (D23);
\draw [<->] (D3) -- (D23);
\draw [->]  (0,4.35) -- (D1);
\draw [<-]  (D23) -- (0,4.35);
\draw [->]  (0,4.35) -- (D3);
\draw [->]  (0,4.35) -- (D12);

\node[circle,draw=white, fill=white, inner sep=0pt,minimum size=10pt] at (0,4.35) {$\hat{\upomega}$};
%\node (omega) at (0,4.35) {$\hat{\upomega}$};
 
\end{tikzpicture}
\caption{Applying $\upphi$, $\hat{\uptheta}$ and $\hat{\upomega}$ to projective indecomposable $H_{\sym_{4}}(0)$-modules.}
\label{fig:PIJtwists}
\end{figure}

%%%%%%%%%%%%%%%%%%%%%%%%%%%%%%%%%%%%%%%%%%%%%%%%%%%%%%%%%

\section{Projective covers and injective hulls}\label{sec:projectivecovers} 

In this section, we determine the projective covers and injective hulls for significant families of $H_W(0)$-modules. Specific applications will be given in Section~\ref{sec:applications}.

Since $H_W(0)$ is Frobenius \cite{Fayers}, this algebra is also Artinian. For Artinian algebras, a submodule $N$ of a module $M$ is \emph{superfluous} if $N$ is contained in the radical ${\rm rad}(M)$ of $M$ (\cite[Lemma 3.4]{Auslander}). For modules $M$ and $K$, an epimorphism $f : M \to K$ is \emph{essential} if $\ker(f)$ is a superfluous submodule of $M$. A \emph{projective cover} of $M$ is a projective module $P$ together with an essential epimorphism $f:P\rightarrow M$. The projective module $P$ is unique up to isomorphism, and we shall refer to $P$, rather than the pair $(P,f)$, as the projective cover of $M$. 

Let $Y$ be an upper order ideal in $\mathcal{D}_I^J$. By Lemma~\ref{lem:quotientinterval}, $\K Y$ is a submodule of $\PP_I^J$. The morphism $f : \PP_I^J \to \PP_I^J/\K Y$ defined by $f(w) = w + \K Y$ is an epimorphism with ${\rm ker}(f) = \K Y$. We will show that if $u_J \notin Y$, then $\PP_I^J$ is the projective cover of $\PP_I^J/\K Y$. 
In \cite[Section 5]{CKNO:projective}, Choi, Kim, Nam and Oh constructed projective covers for the $0$-Hecke modules introduced by Tewari and van Willigenburg in \cite{TvW:2}, in terms of \emph{generalised compositions}, using the ribbon tableau model of \cite{Huang}. Our approach, similarly to \cite{CKNO:projective}, involves directly establishing radical membership; we work with and state results in terms of right descent sets.
 
\begin{lemma}\label{lem:generatornotinJ} Let $I\subseteq J\subseteq S$, and let $Y$ be an upper order ideal in $\mathcal{D}_I^J$ such that $u_J \notin Y$. Let $y \in Y$, and let $s_1, \ldots, s_k \in S$ such that $y = \pi_{s_1} \cdots \pi_{s_k} u_I$ in $\PP_I^J$. Then at least one of $s_1, \ldots, s_k$ is not in $J$.
\end{lemma}
\begin{proof} 
First note that since $u_I\le_L y$, such a sequence $s_1, \ldots,s_k$ exists. Suppose for a contradiction that all of $s_1, \ldots , s_k$ are in $J$. Then $s_1 \cdots s_k\in W_J$, the parabolic subgroup of $W$ generated by $J$. Thus, since $u_I \in W_J$, we have $y = \pi_{s_1} \cdots \pi_{s_k} u_I \in W_J$. Therefore $y \le_L u_J$, and since $Y$ is an upper order ideal, we have  $u_J \in Y$. 
\end{proof}

\begin{theorem} \label{thm:projectivecover} Let $Y$ be an upper order ideal in $\mathcal{D}_I^J$ with $u_J \notin Y$. Then $\PP_I^J$ is the projective cover of $\PP_I^J/\K Y$.
\end{theorem}
\begin{proof} Since $\K Y$ is the kernel of the epimorphism $f: \PP_I^J \to \PP_I^J/\K Y$, it is sufficient to show $\K Y \subseteq \text{rad}(\PP_{I}^{J})$. 
Let $h$ be an isomorphism between $\PP_I^J$ and $\oplus_{I \subseteq X \subseteq J} \PP_X$. Then 
\begin{align*}
h(u_I) = \sum_{I \subseteq X \subseteq J} \sum_{w \in \mathcal{D}_X}a_w w 
\end{align*}
for some coefficients $a_w\in \K$. Let $y \in Y$. Since $u_I$ generates $\PP_I^J$,  there exist $s_1, \ldots , s_k \in S$ such that $y = \pi_{s_1} \cdots \pi_{s_k} u_I$. It follows that 
\begin{align}
h(y) = \pi_{s_1} \cdots \pi_{s_k} h(u_I) = \pi_{s_1} \cdots \pi_{s_k} \left(  \sum_{I \subseteq X \subseteq J} \sum_{w \in \mathcal{D}_X}a_w w  \right).\label{eq:imageofy}
\end{align}

By Lemma~\ref{lem:generatornotinJ} at least one of $s_1, \ldots , s_k$ is not in $J = \D_R(u_J)$, say $s_i$, and thus $s_i\notin \D_R(u_X)$ for all $I\subseteq X\subseteq J$. Since $\D_R(u_X) = \D_L(u_X)$, we have $s_i\notin \D_L(u_X)$. 
Therefore $\pi_{s_1}\cdots \pi_{s_k}u_X \neq u_X$ for all $I \subseteq X \subseteq J$. Moreover, since $u_X$ is the shortest element in $\PP_X$, there is no element $x\neq u_X$ in $\PP_X$ such that $\pi_{s_1}\cdots \pi_{s_k}x = u_X$ in $\PP_X$. 
Thus from \eqref{eq:imageofy} we obtain
\begin{align*}
h(y) =  \sum_{I \subseteq X \subseteq J} \sum_{w \in \mathcal{D}_X}a_w \pi_{s_1} \cdots \pi_{s_k}w  = \sum_{I \subseteq X \subseteq J} \sum_{w \in \mathcal{D}_X \setminus \{u_X\} }\hat{a}_w w, 
\end{align*}
for some coefficients $\hat{a}_w \in \K$.
Since projective indecomposable modules have precisely one maximal submodule (\cite[Lemma 6.1]{Leinster}), it is immediate that ${\rm rad}(\PP_X) = \K([u_X, v_X]_L\setminus \{u_X\})$. Thus $h(y) \in \text{rad}(\oplus_{I \subseteq X \subseteq J}\PP_X)$. Hence $y \in \text{rad}(\PP_I^J)$, and so $\K Y \subseteq \text{rad}(\PP_I^J)$. 
\end{proof}

This result specialises to weak Bruhat interval modules as follows.

\begin{corollary} \label{cor:cover} Let $I \subseteq J$ and $w \in \mathcal{D}_J$. Then $\PP_I^J$ is the projective cover of $\B(u_I, w)$. 
\end{corollary}
\begin{proof}
Let $Y$ denote the set $\mathcal{D}_I^J \setminus [u_I, w]_L$. Then $Y$ is an upper order ideal in $\mathcal{D}_I^J$ by Lemma~\ref{lem:quotientinterval}, and $\PP_I^J/\K Y \cong \B(u_I, w)$. 
Since $I\subseteq J$ we have $u_I\le_L u_J$, and since $w\in \mathcal{D}_J$ we have $u_J \le_L w$. Hence $u_J\in [u_I, w]$, so $u_J\notin Y$. 
Therefore, by Theorem~\ref{thm:projectivecover}, $\PP_I^J$ is the projective cover of $\B(u_I, w)$. 
\end{proof}

Thus for $w\in \mathcal{D}_J$, by \eqref{eq:altdecomp} the projective cover of $\B(u_I, w)$ is indecomposable if and only if $I = J$. 

\begin{remark}
The type A case of Corollary~\ref{cor:cover} has been obtained independently, in the language of generalised compositions, by Kim, Lee and Oh in \cite[Lemma 5.2]{Kim.Lee.Oh}.
\end{remark}

\begin{example} Consider the $H_{\sym_4}(0)$-module $\B(2134, 4132)$, and let $i$ denote $s_i$. Since $2134 = u_{{\{1\}}}$ and $4132 \in \mathcal{D}_{\{1,3\}}$, by Corollary~\ref{cor:cover} we have that $\PP_{\{1\}}^{\{1,3\}}$ is the projective cover of $\B(2134, 4132)$. The projective module $\PP_{\{1\}}^{\{1,3\}}$ is depicted in Figure~\ref{fig:doubledescent}; note the appearance of the interval $[2134, 4132]_L$ in this figure.   
\end{example}

We now use Theorem~\ref{thm:autotwists} to determine the injective hulls of another significant class of $H_W(0)$-modules. 
A proper submodule $N$ of a $H_W(0)$-module $M$ is an \emph{essential submodule} of $M$ if $H \cap N \neq \{0\}$ for all non-zero submodules $H$ of $M$. An \emph{injective hull} of $M$ is an injective module $Q$ together with a monomorphism $g:M\rightarrow Q$ such that the image of $g$ is an essential submodule of $Q$. The injective module $Q$ is unique up to isomorphism, and we shall refer to $Q$, rather than the pair $(Q,g)$, as the injective hull of $M$.

Since $M \mapsto \hat{\upomega}[M]$ is a dual equivalence of categories,  $Q$ is the injective hull (projective cover) of $M$ if and only if ${\hat{\upomega}}[Q]$ is the projective cover (injective hull) of $\hat{\upomega}[M]$. The analogous statement holds for $M\mapsto \hat{\uptheta}[M]$.

\begin{theorem}\label{thm:injectivehull}
Let $Y$ be an upper order ideal in $\mathcal{D}_I^J$ with $v_I \in Y$. Then $\PP_I^J$ is the injective hull of $\K Y$.
\end{theorem}
\begin{proof}
Let $Z = \mathcal{D}_{S\setminus J}^{S \setminus I} \setminus w_0 Y$. Then $Z$ is an upper order ideal in $\mathcal{D}_{S\setminus J}^{S \setminus I}$ and $\mathcal{D}_I^J \setminus w_0 Z = Y$. Hence $\hat{\upomega}[\PP_{S \setminus J}^{S\setminus I} / \K Z] \cong \K Y$ by Theorem~\ref{thm:autotwists}.  
Since $v_I \in Y$, we have $u_{S\setminus I} \notin Z$, and so $\PP_{S \setminus J}^{S\setminus I}$ is the projective cover of $\PP_{S \setminus J}^{S\setminus I} / \K Z$ by Theorem \ref{thm:projectivecover}. 
By Corollary~\ref{cor:PIJtwist}, $\hat{\upomega}[\PP_{S\setminus J}^{S \setminus I}] \cong \PP_I^J$, and therefore $\PP_I^J$ is the injective hull of $\K Y$.
\end{proof}
Note that the monomorphism $g : \K Y \to \PP_I^J$ associated to the injective hull of $\K Y$ is the inclusion map. The specialisation of Theorem~\ref{thm:injectivehull} to weak Bruhat interval modules is as follows.

\begin{corollary} \label{cor:hull} 
Let $I\subseteq J$ and $w \in \mathcal{D}_I$. Then $\PP_I^J$ is the injective hull of $\B(w, v_J)$. 
\end{corollary}

\begin{proof}
Since $w_0w\in \mathcal{D}_{S \setminus I}$, by Corollary \ref{cor:cover} we have that $\PP_{S \setminus J}^{S \setminus I}$ is the projective cover of $\B(u_{S\setminus J}, w_0w)$.  
Therefore $\hat{\upomega}[\PP_{S \setminus J}^{S \setminus I}]$ is the projective cover of 
\begin{equation*}
    \hat{\upomega}[\B(u_{S\setminus J}, w_0w)] \cong \B(w_0w_0w , w_0u_{S\setminus J}) = \B( w, v_{J}),
\end{equation*}
in which the equality is due to Lemma \ref{lem:shortestlongest}(3). By Corollary~\ref{cor:PIJtwist}, $\hat{\upomega}[\PP_{S \setminus J}^{S \setminus I}] \cong \PP_I^J$. 
\end{proof}

%%%%%%%%%%%%%%%%%%%%%%%%%%%%%%%%%%%%%%%%%%%%%%%%%%%%%%%%%%%

\section{Applications to modules for quasisymmetric functions}\label{sec:applications}

The Grothendieck group of finitely-generated $0$-Hecke modules in type A is isomorphic to the ring of quasisymmetric functions via the quasisymmetric characteristic \cite{DKLT}, and much recent work has been devoted to constructing $H_{\sym_n}(0)$-modules whose images under the quasisymmetric characteristic are important families of quasisymmetric functions. In this section, we apply results from Sections~\ref{sec:bruhat} and \ref{sec:projectivecovers} to uniformly recover a number of results on indecomposability, projective covers, and injective hulls for various such modules, and also obtain new results for the modules associated to the recently-introduced row-strict dual immaculate functions and row-strict extended Schur functions of Niese, Sundaram, van Willigenburg, Vega, and Wang \cite{NSvWVW:rowstrict, NSvWVW:modules}.

So far, we have indexed $H_W(0)$-modules by subsets of the generating set $S$ or by intervals in weak Bruhat order. On the other hand, $H_{\sym_n}(0)$-modules associated to quasisymmetric functions are typically indexed by \emph{compositions of $n$}: sequences of positive integers that sum to $n$. These are in bijection with subsets of $[n-1]$, and thus with subsets of the simple generators of $\sym_n$, as follows. If $\alpha = (\alpha_1,\ldots,\alpha_k)$ is a composition of $n$, then the associated subset ${\rm set}(\alpha)$ is $\{\alpha_1,\alpha_1+\alpha_2,\ldots, \alpha_1+\alpha_2 +\ldots + \alpha_{k-1}\}$. We denote the complement of $\set(\alpha)$ by $\set(\alpha)^c$ rather than $[n-1] \setminus {\rm set}(\alpha)$. The \emph{reversal} of $\alpha$, denoted by $\alpha^r$, is obtained by reversing the sequence $\alpha$.
   
\begin{example} Let $\alpha = (1,4,3)$. Then ${\rm set}(\alpha) = \{1,5\}$ and $\alpha^r = (3,4,1)$.
\end{example}
  
In what follows, we index projective indecomposable $H_{\sym_n}(0)$-modules by subsets of $[n-1]$, rather than subsets of $\{s_1, \ldots , s_{n-1}\}$: we let $i$ denote $s_i$. 
First we consider modules for the dual immaculate \cite{BBSSZ:dualimmaculate} and extended Schur \cite{Assaf.Searles} bases of quasisymmetric functions, and their row-strict analogues \cite{NSvWVW:rowstrict, NSvWVW:modules}. 
The \emph{diagram} $D(\alpha)$ associated to a composition $\alpha$ is the left-justified array of boxes with $\alpha_i$ boxes in the $i$th row from the top. A \emph{standard immaculate tableau} of shape $\alpha$ is a labelling of the boxes of $D(\alpha)$ by the integers $1, \ldots, n$, each used once, such that entries increase from left to right along rows and from top to bottom in the first column. A standard immaculate tableau is a $\emph{standard extended tableau}$ if the entries increase from top to bottom in every column. The set of standard immaculate tableaux of shape $\alpha$, and its subset of standard extended tableaux, are denoted by $\SIT(\alpha)$ and $\SET(\alpha)$ respectively. 

Let $T_0^\alpha$ denote the element of $\SET(\alpha)$ (and thus of $\SIT(\alpha)$) obtained by filling the boxes of $D(\alpha)$ with numbers $1, \ldots , n$ consecutively starting with the highest row from left to right, then the second-highest row from left to right, and so on. Let $T_1^\alpha$ denote the element of $\SIT(\alpha)$ obtained by filling the boxes of $D(\alpha)$ with numbers $1, \ldots , n$ consecutively starting with the first column from top to bottom, then the remainder of the lowest row from left to right, then the remainder of the second-lowest row from left to right, and so on. Finally, let $\iT_1^\alpha$ denote the element of $\SET(\alpha)$ obtained by filling the boxes of $D(\alpha)$ with numbers $1, \ldots , n$ consecutively starting with the first column from top to bottom, then the second column from top to bottom, and so on.

\begin{example}\label{ex:SIT}
The standard immaculate tableaux $\SIT(2,2)$ are in Figure~\ref{fig:SIT}. The standard extended tableaux $\SET(2,2)$ are the middle and rightmost tableaux. The leftmost tableau is $T_1^\alpha$, the middle tableau is $\iT_1^\alpha$, and the rightmost tableau is $T_0^\alpha$.    
\end{example}
\begin{figure}[h]
\[
\ytableaushort{14,23} \qquad \ytableaushort{13,24} \qquad \ytableaushort{12,34} 
\]
  \caption{The three standard immaculate tableaux of shape $(2,2)$.}
\label{fig:SIT}
\end{figure}
  
In \cite{BBSSZ:modules}, Berg, Bergeron, Saliola, Serrano and Zabrocki define a $H_{\sym_n}(0)$-action on the $\K$-span of $\SIT(\alpha)$, and show the quasisymmetric characteristics of the resulting modules $\mathcal{V}_\alpha$ are the dual immaculate functions of \cite{BBSSZ:dualimmaculate}. In \cite{Searles:0Hecke}, Searles defines an $H_{\sym_n}(0)$-action on the $\K$-span of $\SET(\alpha)$, and shows the quasisymmetric characteristics of the resulting modules $X_\alpha$ are the extended Schur functions of \cite{Assaf.Searles}. 

Jung, Kim, Lee and Oh \cite{JKLO} identify both $\mathcal{V}_\alpha$ and $X_\alpha$ as weak Bruhat interval modules as follows. For $T \in \SIT(\alpha)$, the \emph{reading word} $\text{rw}(T)$ of $T$ is the permutation obtained from reading the entries in each row in $T$ from right to left, starting with the topmost row and iterating downwards. The isomorphisms
 \begin{equation}
\mathcal{V}_\alpha \cong \B(\text{rw}(T_0^\alpha), \text{rw}(T_1^\alpha)) \,\,\, \mbox{ and } \,\,\, X_\alpha \cong \B(\text{rw}(T_0^\alpha), \text{rw}(\iT_1^\alpha)) \label{eq:VXiso}
\end{equation}
are proved in \cite[Theorem 4.4]{JKLO}. It is also shown in the proof of \cite[Theorem 4.4]{JKLO} that $\text{rw}(T_{0}^{\alpha}) = u_{\text{set}(\alpha)^c}$ and that 
 $\text{rw}(\iT_1^\alpha) \leq_L \text{rw}(T_{1}^{\alpha}) \leq_L  v_{\text{set}(\alpha)^c}$. 
Therefore (as also shown in \cite{CKNO:projective}) $\mathcal{V}_\alpha$ and $X_\alpha$ are quotients of $\PP_{\text{set}(\alpha)^c}$, and $X_\alpha$ is a quotient of $\mathcal{V}_\alpha$. 

Indecomposability of $\mathcal{V}_\alpha$ was established in \cite[Theorem 3.12]{BBSSZ:modules}, and indecomposability of $X_\alpha$ was established in \cite[Theorem 3.13]{Searles:0Hecke}. Proposition~\ref{prop:indecomposablebruhat} in conjunction with \eqref{eq:VXiso} recovers these results, and additionally shows that any quotient of these modules is indecomposable. 

\begin{theorem} For any composition $\alpha$, the modules $\mathcal{V}_\alpha$, $X_\alpha$, and all quotients of these modules are indecomposable.
 \end{theorem}
 \begin{proof}
By \eqref{eq:VXiso}, both $\mathcal{V}_\alpha$ and $X_\alpha$ are weak Bruhat interval modules whose shortest element is the shortest element of a right descent class and whose longest element is in the same right descent class. Therefore, by Proposition~\ref{prop:indecomposablebruhat}, these modules and their quotients are indecomposable.
\end{proof}
 
The projective covers for $\mathcal{V}_\alpha$ and $X_\alpha$ were determined respectively in \cite[Theorem 3.2]{CKNO:projective} and \cite[Theorem 3.5]{CKNO:projective}. One can recover these results via Corollary~\ref{cor:cover}.

 \begin{theorem} 
For any composition $\alpha$, the projective cover of $\mathcal{V}_\alpha$ and $X_\alpha$ is $\PP_{{\rm set}(\alpha)^c}$.
 \end{theorem}
 \begin{proof}
  By \eqref{eq:VXiso}, both $\mathcal{V}_\alpha$ and $X_\alpha$ are weak Bruhat interval modules whose shortest element is the shortest element of the right descent class $\mathcal{D}_{\text{set}(\alpha)^c}$, and whose longest element is also in $\mathcal{D}_{\text{set}(\alpha)^c}$. The statement then follows from Corollary~\ref{cor:cover}. 
 \end{proof}
 
The \emph{row-strict dual immaculate functions} and \emph{row-strict extended Schur functions} \cite{NSvWVW:rowstrict, NSvWVW:modules} are the images of the dual immaculate functions and, respectively, the extended Schur functions under a certain involution on the ring of quasisymmetric functions. In \cite{NSvWVW:modules}, Niese, Sundaram, van Willigenburg, Vega and Wang define a new $H_{\sym_n}(0)$-action on the $\K$-span of $\SIT(\alpha)$, and show the quasisymmetric characteristics of the resulting $H_{\sym_n}(0)$-modules $\mathcal{W}_\alpha$ are the row-strict dual immaculate functions. This action is 
 \begin{align}\label{eq:Waction}
 \pi_i(T) =
  \begin{cases}
                                   T & \text{if $i+1$ is strictly below $i$ in $T$,} \\ 0 & \text{if $i+1$ is in the same row as $i$ in $T$,}   \\
                                   s_i(T) & \text{if $i+1$ is strictly above $i$ in $T$,} 
  \end{cases}
\end{align}
where $\pi_i$ denotes $\pi_{s_i}$, and $s_i(T)$ is the tableau obtained by exchanging the entries $i$ and $i+1$ in $T$. It is moreover shown in \cite{NSvWVW:modules} that the quasisymmetric characteristics of the modules $\mathcal{Z}_\alpha$ resulting from the action \eqref{eq:Waction} on the $\K$-span of $\SET(\alpha)$ are the row-strict extended Schur functions.

\begin{example}\label{ex:W22}
The three elements of $\SIT(2,2)$, along with the $H_{\sym_4}(0)$-action \eqref{eq:Waction} on $\SIT(2,2)$ are shown in Figure~\ref{fig:SITgraph}. 
\end{example}

\begin{figure}[h]
\centering
\begin{tikzpicture}[thick, scale=1]
    \node (1) at (3,4) {$\begin{ytableau} 
      	  1 & 2 \\ 
			3 & 4 
		\end{ytableau}$};
  \node (2) at (0,4) {$\begin{ytableau} 
			1 & 3 \\ 
			2 & 4 
		\end{ytableau}$};
  \node (3) at (-3,4) {$\begin{ytableau} 
    	    1 & 4 \\ 
			2 & 3 
		\end{ytableau}$};		

  \node (30) at (-1.55,2.8) {$0$};
  \node (10) at (4.5,2.8) {$0$};

 \node (r2) at (-2.8,5) {$\pi_1$};
 \node (r2) at (0.4,5) {$\pi_1, \pi_3$};
 \node (r3) at (3.1,5) {$\pi_2$};
 
  \node (s1) at (-1.55,4.25) {$\pi_3$};
  \node (s2) at (1.45,4.25) {$\pi_2$};

  \node (p3*) at (-1.7, 3 .4) {$\pi_2$};
  \node (p2*) at (4.5,3 .4) {$\pi_1, \pi_3$};
  
\draw[->] (3)--(2);
\draw[->] (2)--(1);

\draw[->] (1)--(10);
\draw[->] (3)--(30);

\draw [->] (1) to [out=100,in=155, looseness=3.5] (1);
\draw [->] (2) to [out=100,in=155, looseness=3.5] (2);
\draw [->] (3) to [out=100,in=155, looseness=3.5] (3);

\end{tikzpicture} %%%%%%%%%%%
\caption{The $H_{\sym_4}(0)$-action on $\SIT(2,2)$ defining the module $\mathcal{W}_{(2,2)}$.}
\label{fig:SITgraph}
\end{figure}
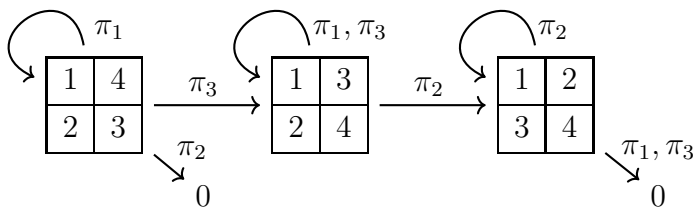

\begin{remark}
We follow \cite{BBSSZ:modules} in referring to the modules for dual immaculate functions as $\mathcal{V}_\alpha$. On the other hand, in \cite{NSvWVW:modules} these modules are referred to as $\mathcal{W}_\alpha$ and the modules for row-strict dual immaculate functions are referred to as $\mathcal{V}_\alpha$. Therefore, our use of $\mathcal{V}_\alpha$ and $\mathcal{W}_\alpha$ is the reverse of how  $\mathcal{V}_\alpha$ and $\mathcal{W}_\alpha$ are used in \cite{NSvWVW:modules}.
\end{remark}

To apply the results of Sections~\ref{sec:bruhat} and \ref{sec:projectivecovers}, we begin by precisely identifying $\mathcal{W}_\alpha$ and $\mathcal{Z}_\alpha$ as weak Bruhat interval modules whose underlying set is a subset of a particular right descent class. For $T\in \SIT(\alpha)$, define the \emph{row-strict reading word} $\rw_{\mathcal R}(T)$ of $T$ to be the permutation obtained by reading the entries of $T$ from left to right along rows, beginning at the bottom row and proceeding to the top row.

\begin{theorem}\label{thm:bruhatrowstrict}
For any composition $\alpha$,
\[\mathcal{W}_\alpha \cong \B(\rwr(T_1^\alpha), \rwr(T_0^\alpha)) \,\,\, \mbox{ and } \,\,\, \mathcal{Z}_\alpha \cong \B(\rwr(\iT_1^\alpha), \rwr(T_0^\alpha)).\] 
Moreover, both of these modules are submodules of $\PP_{\set(\alpha^r)}$.
\end{theorem}
\begin{proof}
We prove this for $\mathcal{W}_\alpha$; the argument for $\mathcal{Z}_\alpha$ is similar. Suppose $\alpha=(\alpha_1, \ldots , \alpha_k)$.
If $T\in \SIT(\alpha)$, then since entries increase along each row and down the first column, $\rwr(T)$ is a permutation that consists of an increasing run of length $\alpha_k$, followed by an increasing run of length $\alpha_{k-1}$, and so on, such that the sequence consisting of the first elements of each increasing run decreases from left to right. 
Conversely, any such permutation is clearly $\rwr(T)$ for some $T\in \SIT(\alpha)$. For any such permutation, right descents occur precisely at the end of each increasing run, hence its right descent set is $\set(\alpha^r)$.

We now show the set of such permutations is precisely the stated interval in left weak Bruhat order. 
For $1\le i < j \le n$, the pair $(i,j)$ is an \emph{ascent pair} for $w\in \sym_n$ if $w(i)<w(j)$. It was shown in the proof of \cite[Lemma 7.11]{Searles:diagramsupermodules} that $u\le_L w$ if and only if every ascent pair of $w$ is also an ascent pair of $u$. For any $T\in \SIT(\alpha)$, every pair $(i,j)$ in $\rwr(T)$ where $w(i),w(j)$ are in the same increasing run is an ascent pair, and every pair $(i,j)$ where $w(i),w(j)$ are the first elements of an increasing run is not an ascent pair. 
Moreover, $\rwr(T_0^\alpha)$ is the permutation such that none of the remaining pairs are ascent pairs, and $\rwr(T_1^\alpha)$ is the permutation such that all of the remaining pairs are ascent pairs. Therefore, $\rwr(T_1^\alpha)\le_L \rwr(T) \le_L \rwr(T_0^\alpha)$ for all $T\in \SIT(\alpha)$. Conversely, any permutation in $\B(\rwr(T_1^\alpha), \rwr(T_0^\alpha))$ must satisfy the above condition on ascent pairs, and hence is $\rwr(T)$ for some $T\in \SIT(\alpha)$, as required. Also, it is easy to see that $\rwr(T_0^\alpha)$ is the longest element of the right descent class $\mathcal{D}_{\set(\alpha^r)}$, and hence $\B(\rwr(T_1^\alpha), \rwr(T_0^\alpha))$ is a submodule of $\PP_{\set(\alpha^r)}$.

Finally we show the $H_{\sym_n}(0)$-action \eqref{eq:Waction} on $\SIT(\alpha)$ agrees with the $H_{\sym_n}(0)$-action on $\B(\rwr(T_1^\alpha), \rwr(T_0^\alpha))$. By definition of $\rwr(T)$, we have $\pi_i(T) = T$ if and only if $s_i\in \D_L(\rwr(T))$. Let $s_i\notin \D_L(\rwr(T))$. Since $i+1$ appears to the right of $i$ in $\rwr(T)$, $i+1$ appears weakly above $i$ in $T$. Now, $s_i \rwr(T)$ is the row-strict reading word of an element of $\SIT(\alpha)$ if and only if $i$ and $i+1$ are in different increasing runs in $\rwr(T)$, that is, if and only if $i+1$ appears strictly higher than $i$ in $T$, since applying $s_i$ to $\rwr(T)$ introduces an additional right descent precisely when $i$ and $i+1$ are in the same increasing run in $\rwr(T)$. Therefore, the $H_{\sym_n}(0)$-actions on $\SIT(\alpha)$ and $\B(\rwr(T_1^\alpha), \rwr(T_0^\alpha))$ agree.
\end{proof}

 \begin{example} In Figure~\ref{fig:SITgraph}, observe that $\rwr(T_0^{(2,2)}) = 3412$ and $\rwr(T_1^{(2,2)}) = 2314$. Hence $\mathcal{W}_{(2,2)}\cong \B(2314, 3412)$.
 \end{example}

The indecomposability of $\mathcal{W}_\alpha$ and $\mathcal{Z}_\alpha$ were established in \cite[Theorem 6.15]{NSvWVW:modules} and \cite[Theorem 7.13]{NSvWVW:modules}. Theorem~\ref{thm:bruhatrowstrict} together with Proposition~\ref{prop:indecomposablebruhat} recovers these results, and additionally shows that any submodule of these modules is indecomposable.

\begin{corollary}
For any composition $\alpha$, the modules $\mathcal{W}_\alpha$, $\mathcal{Z}_\alpha$, and all submodules of these modules are indecomposable.  
\end{corollary}  
 \begin{proof}
 By Theorem~\ref{thm:bruhatrowstrict}, $\mathcal{W}_\alpha$ and $\mathcal{Z}_\alpha$ are weak Bruhat interval modules whose longest element is the longest element of a right descent class and whose shortest element is in the same right descent class. Therefore, by Proposition~\ref{prop:indecomposablebruhat}, these modules and their submodules are indecomposable.
 \end{proof}
 
Using Corollary~\ref{cor:hull}, we determine the injective hulls of $\mathcal{W}_\alpha$ and $\mathcal{Z}_\alpha$.

\begin{corollary}
For any composition $\alpha$, the injective hull of $\mathcal{W}_\alpha$ and $\mathcal{Z}_\alpha$ is $\PP_{{\rm set}(\alpha^r)}$.    
\end{corollary}
\begin{proof}
By Theorem~\ref{thm:bruhatrowstrict}, $\mathcal{W}_\alpha$ and $\mathcal{Z}_\alpha$ are weak Bruhat interval modules whose longest element is the longest element of the right descent class $\mathcal{D}_{\set(\alpha^r)}$, and whose shortest element is also in $\mathcal{D}_{\set(\alpha^r)}$.
The statement then follows from Corollary~\ref{cor:hull}.
\end{proof}

\begin{remark}
One could replace the first two paragraphs of the proof of Theorem~\ref{thm:bruhatrowstrict} by noting that $\rwr(T) = \rw(T)w_0$ for all $T\in \SIT(\alpha)$ and appealing to \eqref{eq:VXiso}. However, we wished to demonstrate how this structure could be determined directly; this same method could alternatively be used to prove \eqref{eq:VXiso}. 
Additionally, the fact that $\rwr(T) = \rw(T)w_0$, in conjunction with Theorem~\ref{thm:bruhatrowstrict}, implies that $\mathcal{W}_\alpha \cong \hat{\uptheta}[\mathcal{V}_{\alpha}]$ and similarly $\mathcal{Z}_\alpha \cong \hat{\uptheta}[X_\alpha]$. This provides an alternative way to obtain indecomposability and injective hulls for $\mathcal{W}_\alpha$ and $\mathcal{Z}_\alpha$. 
 We note the fact that the modules for row-strict dual immaculate and row-strict extended Schur functions can be obtained by applying $\hat{\uptheta}$ to the modules for dual immaculate and extended Schur functions is observed in \cite[Table 4.1]{Choi.Kim.Oh:23}. 
\end{remark}

For completeness, we also give the projective cover of $\mathcal{W}_\alpha$. Choi, Kim, Nam, and Oh showed that the injective hull of $\mathcal{V}_\alpha$ is $\oplus_{\beta \in [\ul{\bb\alpha}]} \PP_{\text{set}(\beta)^c}$ \cite[Theorem 4.1]{CKNO:homological}, where $[\ul{\bb\alpha}]$ is a particular set of compositions obtained from $\alpha$; see \cite[Section 4]{CKNO:homological} for a full definition of $[\ul{\bb\alpha}]$.

\begin{theorem}
For any composition $\alpha$, the projective cover of $\mathcal{W}_\alpha$ is $\oplus_{\beta \in [\ul{\bb\alpha}]} \PP_{{\rm set}(\beta^r)}$. 
\end{theorem}
\begin{proof} 
Since $M \mapsto \hat{\uptheta}[M]$ is a dual equivalence of categories and $\mathcal{W}_\alpha \cong \hat{\uptheta}[\mathcal{V}_{\alpha}]$, the projective cover of $\mathcal{W}_\alpha$ is $\hat{\uptheta}[\oplus_{\beta \in [\ul{\bb\alpha}]}\PP_{{\rm set}(\beta)^c}]$. One obtains
\begin{align*}
    \hat{\uptheta}[\oplus_{\beta \in [\ul{\bb\alpha}]} \PP_{{\rm set}(\beta)^c}] = \oplus_{\beta \in [\ul{\bb\alpha}]} \hat{\uptheta}[\PP_{{\rm set}(\beta)^c}] \cong \oplus_{\beta \in [\ul{\bb\alpha}]} \PP_{(w_0{\rm set}(\beta)^c w_0)^c} = \oplus_{\beta \in [\ul{\bb\alpha}]} \PP_{{\rm set}(\beta^r)},
\end{align*}
where the isomorphisms follow from Corollary~\ref{cor:PIJtwist} and the fact that $(w_0{\rm set}(\beta)^c w_0)^c = {{\rm set}(\beta^r)}$, with $i\in {\rm set}(\beta)^c$ understood as $s_i$ for the purpose of conjugating by $w_0$.
\end{proof}
To our knowledge, the injective hull of $X_\alpha$ and projective cover of $\mathcal{Z}_\alpha$ have not yet been determined.

Finally we consider modules for the quasiysmmetric Schur functions \cite{HLMvW11:qs}, which were defined on \emph{standard reverse composition tableaux} by Tewari and van Willigenburg in \cite{TvW:1}. These modules were generalised by Tewari and van Willigenburg in \cite{TvW:2} to modules $\textbf{S}_\alpha^\sigma$ defined on \emph{standard permuted composition tableaux}. Here $\alpha$ is a composition and $\sigma$ a permutation; see \cite[Section 3]{TvW:2} for a full definition of these modules. The modules $\textbf{S}_\alpha^\sigma$ decompose as a direct sum of submodules $\textbf{S}_\alpha^\sigma = \oplus_E\textbf{S}_{\alpha, E}^\sigma$, where each $E$ is an equivalence class of standard permuted composition tableaux. Each of the submodules $\textbf{S}_{\alpha, E}^\sigma$ is indecomposable; this was proved for $\sigma = {\rm id}$ by K\"onig \cite[Theorem 4.11]{Konig}, and in general by Choi, Kim, Nam and Oh \cite[Theorem 3.1]{CKNO:tableaux}. 

Jung, Kim, Lee and Oh define a reading word ${\text{rw}_{\mathcal{S}}}$ on the standard permuted composition tableaux (\cite[Definition 4.6]{JKLO}). Let $\tau_E$ (respectively, $\tau'_E$) denote the standard permuted composition tableau in $E$ that has shortest (respectively, longest) reading word. It is proved in \cite[Theorem 4.8]{JKLO} that 
\begin{align}
\textbf{S}_{\alpha,E}^\sigma \cong \B({\text{rw}_{\mathcal{S}}}(\tau_E),{\text{rw}_{\mathcal{S}}}(\tau'_E)), \label{eq:qschurbruhatinterval}
\end{align}
and that ${\text{rw}_{\mathcal{S}}}(\tau_E)$ is the shortest element of some right descent class. We note however that these weak Bruhat interval modules typically contain elements from more than one right descent class, and therefore Proposition~\ref{prop:indecomposablebruhat} does not apply.
   
The projective cover of $\textbf{S}_{\alpha,E}^\sigma$  was determined in \cite[Theorem 5.11]{CKNO:projective} in terms of a generalised composition associated to $E$. 
Since $\textbf{S}_{\alpha,E}^\sigma$ is a weak Bruhat interval module whose shortest element is the shortest element of a right descent class \eqref{eq:qschurbruhatinterval}, Corollary~\ref{cor:cover} recovers this result, with a different statement in terms of the right descent sets of the reading words of the tableaux $\tau_E$ and $\tau'_E$. 
 
\begin{theorem}\label{thm:qschurcover} 
Suppose ${\textup{rw}_{\mathcal{S}}}(\tau_E) \in \mathcal{D}_I$ and ${\textup{rw}_{\mathcal{S}}}(\tau_E') \in \mathcal{D}_J$. Then $\PP_{I}^J$ is the projective cover of $\textbf{S}_{\alpha,E}^\sigma$.
\end{theorem}
\begin{proof} From \eqref{eq:qschurbruhatinterval} we have that $\textbf{S}_{\alpha,E}^\sigma \cong \B(u_I, w)$ for some $w \in \mathcal{D}_J$ and $I \subseteq S$. Therefore $\PP_{I}^J$ is the projective cover of $\textbf{S}_{\alpha,E}^\sigma$ by Corollary~\ref{cor:cover}.
\end{proof}

The images of the modules ${\bf S}_\alpha^\sigma$ and ${\bf S}_{\alpha, E}^\sigma$ under $\hat{\upomega}$ are a family of modules that generalise the modules introduced in \cite{Bardwell.Searles} for the Young row-strict dual immaculate functions of \cite{Mason.Niese}. Specifically, denoting these modules by $\textbf{R}_\alpha^{\sigma}$ and $\textbf{R}_{\alpha, E}^{\sigma}$, one has $\textbf{R}_\alpha^{\sigma} \cong \hat{\upomega}[\textbf{S}^{\sigma^{w_0}}_{\alpha^r}]$ and, when $E$ is an equivalence class of standard permuted composition tableaux corresponding to $\alpha_r$ and $\sigma^{w_0}$,  $\textbf{R}_{\alpha, E}^{\sigma} \cong \hat{\upomega}[\textbf{S}^{\sigma^{w_0}}_{\alpha^r, E}]$ (\cite[Proposition 4.11]{JKLO}).
The injective hull of $\textbf{R}_{\alpha,E}^{\sigma}$ was determined in \cite[Corollary 4.13]{JKLO}, via $\hat{\upomega}$. Applying $\hat{\upomega}$ to the statement of Theorem~\ref{thm:qschurcover} yields the following description of the injective hull in terms of the right descent sets of the reading words of $\tau_E$ and $\tau_E'$.

\begin{corollary} 
Let $E$ be an equivalence class of standard permuted composition tableaux corresponding to $\alpha^r$ and $\sigma^{w_0}$. Suppose ${\textup{rw}_{\mathcal{S}}}(\tau_E) \in \mathcal{D}_I$ and ${\textup{rw}_{\mathcal{S}}}(\tau_E') \in \mathcal{D}_J$. Then $\PP_{S \setminus J}^{S \setminus I}$ is the injective hull of $\textbf{R}_{\alpha,E}^{\sigma}$.
\end{corollary}
\begin{proof}
By Theorem~\ref{thm:qschurcover} we have that $\PP_I^J$ is the projective cover of $\textbf{S}_{\alpha^r,E}^{\sigma^{w_0}}$. Therefore $\hat{\upomega}[\PP_{I}^J]$ is the injective hull of $\textbf{R}_{\alpha,E}^{\sigma}$, and  $\hat{\upomega}[\PP_{I}^J] \cong \PP_{S\setminus J}^{S \setminus I}$ by Corollary~\ref{cor:PIJtwist}.
\end{proof} 

%%%%%%%%%%%%%%%%%%%%%%%%%%%%%%%%%%%%%%%%%%%%%%%%%%

\section*{Acknowledgements}
The authors are grateful to Jia Huang and Woo-Seok Jung for helpful conversations. 
The authors were supported by the Marsden Fund, administered by the Royal Society of New Zealand Te Ap{\= a}rangi.

\bibliographystyle{abbrv} 
\bibliography{finitetype}

\end{document}